\documentclass[10pt,a4paper]{amsart}

\usepackage[all,cmtip]{xy}
\usepackage{amssymb}
\usepackage{xcolor}

\theoremstyle{plain}
\newtheorem{theorem}{Theorem}[section]

\newtheorem{corollary}[theorem]{Corollary}
\theoremstyle{definition}

\theoremstyle{remark}
\newtheorem{remark}[theorem]{Remark}

\def\bin #1#2 {\left( \matrix { #1 \cr #2 \cr } \right) }

\begin{document}

\title[Nilpotent cone and bivariant theory]
{Nilpotent cone and bivariant theory}

\author{Vincenzo Di Gennaro }
\address{Universit\`a di Roma \lq\lq Tor Vergata\rq\rq, Dipartimento di Matematica,
Via della Ricerca Scientifica, 00133 Roma, Italy.}
\email{digennar@axp.mat.uniroma2.it}
%\thanks{}

\author{Davide Franco }
\address{Universit\`a di Napoli
\lq\lq Federico II\rq\rq, Dipartimento di Matematica e Applicazioni
\lq\lq R. Caccioppoli\rq\rq, Via Cintia, 80126 Napoli, Italy.}
\email{davide.franco@unina.it}

%    Information for third author
\author{Carmine Sessa }
\address{Universit\`a di Napoli
\lq\lq Federico II\rq\rq, Dipartimento di Matematica e Applicazioni
\lq\lq R. Caccioppoli\rq\rq, Via Cintia, 80126 Napoli, Italy.}
\email{carmine.sessa2@unina.it}

\abstract We exhibit a new proof, relying on bivariant theory,
that the nilpotent cone is rationally smooth. Our approach enables
us to prove a slightly more general statement.



\endabstract
\maketitle

\section{Introduction}

In \cite{BM} Borho and MacPherson proved that the nilpotent cone
is a rational homology manifold. The proof relies on the celebrated
Decomposition Theorem by Beilinson, Bernstein, Deligne and Gabber
\cite{BBD} and on the Springer's theory of Weyl group representations
(see \cite{BM} and the references therein).

The aim of this paper is to present a new proof, in our opinion
conceptually very simple, based on the bivariant  theory founded by
Fulton and MacPherson in \cite{FultonCF}. Actually, our approach
enables us to prove a slightly more general  statement (see Remark
\ref{r2} below). By {\it bivariant theory} we intend  {\it the
topological bivariant homology theory with coefficients in a
Noetherian  commutative ring with identity $\mathbb A$}  \cite[pp.
32, 83 and p. 86, Corollary 7.3.4]{FultonCF}.

That the nilpotent cone is a rational homology manifold can be seen
as an easy consequence of a characterization of homology manifolds
we recently proved in \cite[Theorem 6.1]{DFS}: {\it given a
resolution of singularities $\pi:\mathcal{\widetilde N}\to
\mathcal{N}$ of a quasi-projective variety $\mathcal{N}$, then
$\mathcal{N}$ is a homology manifold if and only if there exists a
bivariant class of degree one for $\pi$}. A {\it bivariant class of
degree one for $\pi$} is an element $\eta\in H^0(\mathcal{\widetilde
N}\stackrel{\pi}\to \mathcal{N})$ such that the induced Gysin
homomorphism $\eta_0:H^0(\mathcal{\widetilde N})\to
H^0(\mathcal{N})$ sends $1_{\mathcal{\widetilde N}}$ to
$1_{\mathcal{N}}$.

\section{The main result}

\begin{theorem}\label{main}
Let $\pi':{\mathbf {\widetilde g}}\to {\mathbf {g}}$ be a projective
morphism between complex quasi-projective nonsingular varieties of
the same dimension. Assume that $\pi'$ is generically finite, of
degree $\delta$. Let $\mathcal N\subset {\mathbf {g}}$ be a closed
irreducible subvariety. Consider the induced  fibre square diagram:
 \[
\xymatrix{
\mathcal{\widetilde N}  \ar[d]^{\pi} \ar@{^{(}->}[r]& {\mathbf {\widetilde g}} \ar[d]^{\pi'}\\
\mathcal N \ar@{^{(}->}[r]^i &{\mathbf {g}},
}
\]
where $\mathcal{\widetilde N}:=\mathcal N\times_{\mathbf {g}}{\mathbf {\widetilde g}}$. If 
$\mathcal{\widetilde N}$ is irreducible and nonsingular and $\pi$ is birational,
then $\mathcal{N}$ is an $\mathbb A$-homology manifold for every
Noetherian commutative ring with identity $\mathbb A$ for which $\delta$ is a unit.
\end{theorem}

\begin{proof}

Since $\pi':{\mathbf {\widetilde g}}\to {\mathbf {g}}$ is a
projective morphism between complex quasi-projective nonsingular
varieties of the same dimension, it is a local complete intersection
morphism of relative codimension $0$ \cite[p. 130]{FultonCF}. Let
$$\theta' \in H^0({\mathbf {\widetilde g}}\stackrel{\pi'}\to {\mathbf
{g}})\cong Hom_{D^{b}_{c}({\mathbf {g}})}(R{\pi'}_*\mathbb
A_{{\mathbf {\widetilde g}}}, \mathbb A_{{\mathbf {g}}})$$ be the
orientation class of $\pi'$ \cite[p. 131]{FultonCF}. Let
$\theta'_0:H^0({\mathbf {\widetilde g}})\to H^0({\mathbf {g}})$ be
the induced Gysin map. It is clear that $\theta'_0(1_{\mathbf
{\widetilde g}})=\delta\cdot 1_{\mathbf {g}}\in H^0({\mathbf {g}})$,
where $\delta$ is the degree of $\pi'$. Therefore, if we denote by
$$\theta:=i^*\theta'\in H^0(\mathcal{\widetilde N}\stackrel{\pi}\to \mathcal{N})\cong
Hom_{D^{b}_{c}(\mathcal N)}(R{\pi}_*\mathbb A_{\mathcal {\widetilde
N}}, \mathbb A_{\mathcal N})$$ the pull-back of $\theta'$, then
$\delta^{-1}\cdot \theta$ is a bivariant class of degree one for
$\pi$ \cite[2. Notations, $(ii)$]{DFS}. At this point, our claim
follows by \cite[Theorem 6.1]{DFS}. \textit{For the Reader's convenience, let us briefly summarize the argument}.

\medskip
Since $\delta^{-1}\cdot \theta$ is a bivariant class of degree one
for $\pi$, it follows that $\left(\delta^{-1}\cdot
\theta\right)\circ \pi^*={\text{id}}_{\mathbb A_\mathcal N}$ in
${D^{b}_{c}(\mathcal N)}$, i.e. that $\delta^{-1}\cdot \theta$ is a
section of the pull-back $\pi^*:\mathbb A_{\mathcal{N}}\to
R\pi_*\mathbb A_{\mathcal{\widetilde N}}$ \cite[Remark 2.1,
$(i)$]{DFS}. Hence, $\mathbb A_{\mathcal{N}}$ is a direct summand of
$Rf_*\mathbb A_{\mathcal{\widetilde N}} $ in ${D^{b}_{c}(\mathcal
N)}$ \cite[Lemma 3.2]{DFS} and so we have a decomposition
\begin{equation}\label{first}
Rf_*\mathbb A_{\mathcal{\widetilde N}} \cong \mathbb A_{\mathcal{
N}} \oplus \mathcal K.
\end{equation}
Now, set $\nu=\dim {\mathcal{\widetilde
N}}=\dim{\mathcal{N}}$ and let $[{\mathcal{\widetilde N}}]\in
H_{2\nu}({\mathcal{\widetilde N}})$ be the fundamental class of
${\mathcal{\widetilde N}}$. We have:
$$
[{\mathcal{\widetilde N}}]\in H_{2\nu}({\mathcal{\widetilde
N}})\cong H^{-2\nu}({\mathcal{\widetilde N}}\stackrel{}\to pt.)\cong
Hom_{D^{b}_{c}({\mathcal{\widetilde N}})}(\mathbb
A_{\mathcal{\widetilde N}}[\nu],D\left(\mathbb
A_{\mathcal{\widetilde N}}[\nu]\right)),
$$
where $D$ denotes Verdier dual. Therefore, $[{\mathcal{\widetilde
N}}]$ corresponds to a morphism
\begin{equation}\label{is}
\mathbb A_{\mathcal{\widetilde N}}[\nu]\to D\left(\mathbb
A_{\mathcal{\widetilde N}}[\nu]\right),
\end{equation}
whose induced map in hypercohomology is nothing but the duality
morphism
\begin{equation}\label{dm}
\mathcal D_{\mathcal{\widetilde N}}: x\in
H^{\bullet}({\mathcal{\widetilde N}})\to x\cap [{\mathcal{\widetilde
N}}]\in H_{2\nu-\bullet}({\mathcal{\widetilde N}}).
\end{equation}
If we assume that ${\mathcal{\widetilde N}}$ is nonsingular (actually it suffices
that ${\mathcal{\widetilde N}}$ is an $\mathbb
A$-homology manifold, the morphisms (\ref{is}) and (\ref{dm}) are isomorphisms.
The first one induces an isomorphism
\begin{equation*}
R\pi_*\mathbb A_{\mathcal{\widetilde N}}[\nu]\to D\left(R\pi_*\mathbb A_{\mathcal{\widetilde N}}[\nu]\right),
\end{equation*}
which in turn, via decomposition (\ref{first}),
induces two projections
\begin{equation}\label{proiez}
\mathbb A_{\mathcal{N}}[\nu] \to D\left(\mathbb
A_{\mathcal{N}}[\nu]\right), \quad  \mathcal K[\nu]\to D\left(
\mathcal K[\nu]\right).
\end{equation}
Making explicit the isomorphism induced in cohomology and homology
by (\ref{first}), one may prove \cite[Corollary 5.1]{DFS} that $\mathcal
D_{\mathcal{\widetilde N}}$ is the direct sum of $P_1$ and $P_2$,
where
$$P_1:
H^{\bullet}({\mathcal{N}})\to H_{2\nu-\bullet}({\mathcal{N}}) \quad
{\text{and}}\quad P_2:\mathbb H(\mathcal K[\nu])\to \mathbb
H(D\left( \mathcal K[\nu]\right))
$$
are the maps induced in hypercohomology by the projections
(\ref{proiez}). It follows that $P_1$ is
an isomorphism, because so is $\mathcal D_{\mathcal{\widetilde N}}$, and this holds true when restricting to every open subset 
$U$ of $\mathcal{ N}$. For instance (see also \cite[Corollary 5.1]{DFS}), if $\widetilde U = \pi^{-1}(U)$, the vanishing of the morphism $\mathbb H^{\bullet}(\mathcal K_U[\nu]) \to H_{2\nu-\bullet}(U)$ derives from projection formula \cite[p. 26, G4, (ii)]{FultonCF}:
$$
\pi_*([\widetilde U]\cap \lambda_*  w)=\pi_*(\delta^{-1}\theta^*[U]\cap \lambda_*
 w)=\delta^{-1}(\theta_*\lambda_* w)\cap [U]=0, \quad \forall w \in \mathbb H^{\bullet}(\mathcal K_U[\nu]),
$$
where $\lambda_*$ is the morphism induced in hypercohomology by $\mathcal K_U[\nu]\to R\pi_*\mathbb A_{\widetilde U}[\nu]$.

Therefore, we have $\mathbb A_{\mathcal{N}}[\nu] \cong
D\left(\mathbb A_{\mathcal{N}}[\nu]\right)$, which is equivalent to
say that $\mathcal N$ is an $\mathbb A$-homology manifold.
\end{proof}

%Denote by
%$$P_1:
%H^{\bullet}({\mathcal{N}})\to H_{2\nu-\bullet}({\mathcal{N}}) \quad {\text{and}}\quad
%P_2:\mathbb H(\mathcal
%K[\nu])\to \mathbb H(D\left( \mathcal K[\nu]\right))
%$$
%the maps induced in
%hypercohomology.
%Using projection formulas  \cite[p. 24, p. 26 G4, (ii) and (iii)]{FultonCF}
%(compare with the proof of \cite[Corollary 5.1]{DFS}), one proves that
%$\mathcal D_{\mathcal{\widetilde N}}$ is the direct sum of $P_1$ and $P_2$ (and that
%$P_1$ is the duality morphism $\mathcal D_{\mathcal{N}}$ of ${\mathcal{N}}$).
%It follows that $P_1$ is an isomorphism, because so is $\mathcal D_{\mathcal{\widetilde N}}$.
%And this holds true when restricting to every open
%subset of $\mathcal N$. Therefore, we have
%$\mathbb A_{\mathcal{N}}[\nu] \cong D\left(\mathbb A_{\mathcal{N}}[\nu]\right)$.
%This is equivalent to say
%that $\mathcal N$ is an $\mathbb A$-homology manifold.
%\end{proof}

\begin{remark}
Observe that, as a scheme, $\mathcal{\widetilde N}$ could also be
nonreduced, but what matters is that, for the usual topology,  it is
a nonsingular variety \cite[p. 32, 3.1.1]{FultonCF}
\end{remark}

\begin{corollary}
The nilpotent cone is a rational homology manifold.
\end{corollary}

\begin{proof}
Let $\pi: \mathcal{\widetilde N}\to \mathcal{N}$ be the Springer
resolution of the nilpotent cone $\mathcal{N}$. It extends to a
generically finite projective morphism $\pi':{\mathbf {\widetilde g}}\to {\mathbf { g}}$, known as the Grothendieck simultaneous resolution, between complex
quasi-projective nonsingular varieties of the same dimension
 \cite[p. 49]{BM}.
Therefore, Theorem \ref{main} applies.
\end{proof}

\begin{remark}\label{r2}
If the Grothendieck simultaneous resolution $\pi': {\mathbf{\widetilde{g}}}\to {\mathbf { g}}$ has degree $\delta$, by
Theorem \ref{main} we deduce that {\it the nilpotent cone $\mathcal N$ is
an $\mathbb A$-homology manifold for every Noetherian commutative
ring with identity $\mathbb A$ for which $\delta$ is a unit}. For
instance, for the variety $\mathcal N$ of nilpotent matrices in
${\text{GL}}(n,\mathbb C)$, we have $\delta=n!$. Therefore, in this case,
$\mathcal N$ is also a $\mathbb Z_h$-homology manifold for every
integer $h$ relatively prime with $n!$ in $\mathbb Z$.
\end{remark}

\smallskip
{\bf{Statements and Declarations.}}

\medskip
{\bf Funding.} The authors declare that no funds, grants, or other
support were received during the preparation of this manuscript.

\medskip
{\bf Data availability.} Data sharing not applicable to this article
as no datasets were generated or analysed during the current study.

\medskip
{\bf Competing Interests.} The authors have no relevant financial or
non-financial interests to disclose.

\end{document}